\documentclass[11pt]{article}%
\usepackage{amsmath}
\usepackage{amssymb}
\usepackage{amsfonts}
\usepackage{graphicx}%
\providecommand{\U}[1]{\protect\rule{.1in}{.1in}}
\newtheorem{theorem}{Theorem}

\newtheorem{lemma}[theorem]{Lemma}

\newtheorem{proposition}[theorem]{Proposition}
\newtheorem{remark}[theorem]{Remark}

\newenvironment{proof}[1][Proof]{\noindent\textbf{#1.} }{\ \rule{0.5em}{0.5em}}
\begin{document}

\title{On Riemann surfaces of genus $g$ with $4g$ automorphisms}
\author{Emilio Bujalance\\{\small Departamento de Matem\'{a}ticas Fundamentales, }\\{\small UNED, }\\{\small Paseo Senda del Rey 9, }\\{\small 28040 Madrid, Spain}
\and Antonio F. Costa\\{\small Departamento de Matem\'{a}ticas Fundamentales, }\\{\small UNED, }\\{\small Paseo Senda del Rey 9, }\\{\small 28040 Madrid, Spain}
\and Milagros Izquierdo\\{\small Matematiska institutionen}\\{\small Link{\"o}pings universitet}\\{\small 581 83 Link{\"o}ping, Sweden}\\{\small miizq@mai.liu.se}}
\date{}
\maketitle

\begin{abstract} 
We determine, for all genus $g\geq2$ the Riemann surfaces
of genus $g$ with $4g$ automorphisms. For $g\neq$ $3,6,12,15$ or $30$, this
surfaces form a real Riemann surface $\mathcal{F}_{g}$ in the moduli space
$\mathcal{M}_{g}$: the Riemann sphere with three punctures. The set of real Riemann surfaces in $\mathcal{F}_{g}$
consists of three intervals its closure in the Deligne-Mumford
compactification of $\mathcal{M}_{g}$ is a closed Jordan curve.
\end{abstract}


\section{Introduction}

Given a linear expression like $ag+b$, where $a,b$ are fixed integers, it is
very difficult to claim precise information on the (compact)\ Riemann surfaces
of genus $g\geq2$ with automorphism groups of order $ag+b$: i.e. are there
Riemann surfaces in these conditions?, how many?, which are their automorphism
groups?\ For instance, there are many works about Hurwitz surfaces, i. e.
surfaces of genus $g$ with group of automorphisms of order $84g-84$ (maximal
order), but there is no a complete answer to the above questions. Surprisingly
we shall give an almost complete answer (up to a finite number of genera $g$)
to all questions on Riemann surfaces of genus $g$ with $4g$ automorphisms.

For each integer $g\geq2$ we find an equisymmetric (complex)-\-uni\-parametric
family $\mathcal{F}_{g}$ of Riemann surfaces of genus $g$ having (full)
automorphism group of order $4g$. The families $\mathcal{F}_{g}$ are the
equisymmetric and uniparametric families of Riemann surfaces whose
automorphism groups have largest order. If $g\neq3,6,15$ all surfaces with
$4g$ automorphisms are in the family $\mathcal{F}_{g}$ with one or two more
exceptional surfaces in a few genera: $g=3,6,12,30$. For genera
$g=3,6$ and $15$ it appears another exceptional uniparametric family. Finally
for genera $3,6,12$ and $30$ there are one or two exceptional surfaces with
$4g$ automorphisms.

The automorphism group of the surfaces in $\mathcal{F}_{g}$ is $D_{2g}$ and
the quotient $X/\mathrm{Aut}(X)$ is the Riemann sphere $\widehat{\mathbb{C}}$,
the meromorphic function $X\rightarrow X/\mathrm{Aut}(X)=\widehat{\mathbb{C}}$
have four singular values of orders $2,2,2,2g$.

Ravi S. Kulkarni \cite{K-1} showed that, for any genus
$g\equiv0,1,2\operatorname{mod}4$, there is a unique surface of genus $g$ with
full automorphism group of order $8(g+1)$ (the family of Accola-Maclachan
\cite{A} and \cite{M}), and for $g\equiv-1\operatorname{mod}4,$ there is just
another surface of genus $g$ (the Kulkarni surface \cite{K-1}). In \cite{K}
Kulkarni shows that, if $g\neq3$ there is a unique Riemann surface of
genus $g$ admitting an automorphism of order $4g$, while for $g=3$ there are
two such surfaces. The surfaces in this last family have exactly $8g$
automorphisms, except for $g=2,$ where the surface has $48$ automorphisms. For
cyclic groups there are some cases where the order of the group determines the
Riemann surface (see \cite{K}, \cite{Nak}, \cite{Hi}). Analogous results are
known for Klein surfaces: \cite{BC}, \cite{BT}, \cite{BTG} and \cite{B}.

The family $\mathcal{F}_{g}$ contains surfaces admitting anticonformal
automorphisms, forming the subset $\mathbb{R}\mathcal{F}_{g}$. These points in
the moduli space correspond to Riemann surfaces given by the
complexification of real algebraic curves. The extended groups of
automorphisms of the surfaces in $\mathbb{R}\mathcal{F}_{g}$ (including the
anticonformal automorphisms) are isomorphic either to $D_{2g}\times C_{2}$ or
$D_{4g}$, and such groups contain anticonformal involutions, so the surfaces
in $\mathbb{R}\mathcal{F}_{g}$ are real Riemann surfaces. The topological
types of conjugacy classes of anticonformal involutions (real forms)\ of the
real Riemann surfaces in $\mathcal{F}_{g}$ are either $\{+2,0,-2,-2\}$,
$\{-1$,$-1,-g,-g\}$, $\{0,0,-2,-2\}$ if $g$ is odd or $\{+1,0,-1,-3\}$,
$\{-1$,$-1,-g,-g\}$, $\{-2\}$ if $g$ is even.

The family $\mathcal{F}_{g}$ is the Riemann sphere with three punctures, having an
anticonformal involution whose fixed point set consists of three arcs
$a_{1},a_{2},b$. Each one of these arcs is formed by the real Riemann surfaces
in $\mathbb{R}\mathcal{F}_{g}$ with a different set of topological types of
real forms. Adding three points to the surface $\mathcal{F}_{g}$ we obtain a
compact Riemann surface $\overline{\mathcal{F}_{g}}\subset\widehat{\mathcal{M}%
}_{g}$, where $\overline{a_{1}\cup a_{2}\cup b}$ (the closure of $a_{1}\cup
a_{2}\cup b$ in $\widehat{\mathcal{M}}_{g}$) is a closed Jordan curve. The
space $\widehat{\mathcal{M}}_{g}$ is the Mumford-Deligne compactification of
$\mathcal{M}_{g}$. As a consequence we have that $\overline{\mathbb{R}%
\mathcal{F}_{g}}\cap\mathcal{M}_{g}$ has two connected components.

\textbf{Acknowledgement.} All authors partially supported by the project MTM2014-55812-P.

\section{Preliminaries}

\subsection{Non-Euclidean crystallographic groups}

A \textit{non-Euclidean crystallographic group} (or \textit{NEC group})
$\Gamma$ is a discrete group of isometries of the hyperbolic plane
$\mathbb{D}$. We shall assume that an NEC group has a compact orbit space. If
$\Gamma$ is such a group then its algebraic structure is determined by its
signature%
\begin{equation}
(h;\pm;[m_{1},\ldots,m_{r}];\{(n_{11},\ldots,n_{1s_{1}}),\ldots,(n_{k1}%
,\ldots,n_{ks_{k}})\}). \tag{1}%
\end{equation}
\label{signature}

The orbit space $\mathbb{D}/\Gamma$ is a surface, possibly with boundary. The
number $h$ is called the\textit{\ genus} of $\Gamma$ and equals the
topological genus of $\mathbb{D}/\Gamma$, while $k$ is the number of the
boundary components of $\mathbb{D}/\Gamma$, and the sign is $+$ or $-$
according to whether the surface is orientable or not. The integers $m_{i}%
\geq2$, called the \textit{proper periods}, are the branch indices over
interior points of $\mathbb{D}/\Gamma$ in the natural projection
$\pi:\mathbb{D}\rightarrow\mathbb{D}/\Gamma$. The bracketed expressions
$(n_{i1},\ldots,n_{is_{i}})$, some or all of which may be empty (with
$s_{i}=0$), are called the \textit{period cycles} and represent the branchings
over the ${i}^{\mathrm{th}}$ boundary component of the surface. Finally the
numbers $n_{ij}\geq2$ are the \textit{link periods}.

Associated with each signature there exists a \textit{canonical presentation}
for the group $\Gamma$. If the signature (\ref{signature}) has sign $+$ then
$\Gamma$ has the following generators:

\qquad$x_{1}, \ldots, x_{r}$ \ (elliptic elements),

\qquad$c_{10}, \ldots, c_{1s_{1}}, \ldots, c_{k0}, \ldots, c_{ks_{k}} $ \ (reflections),

\qquad$e_{1},\ldots,e_{k}$ \ (boundary transformations),

\qquad$a_{1}, b_{1}, \ldots, a_{g},b_{g}$ \ (hyperbolic elements);

these generators satisfy the defining relations

\qquad$x_{i}^{m_{i}} = 1$ \ (for $1 \leq i \leq r$),

\qquad$c_{ij-1}^{2}=c_{ij}^{2}= (c_{ij-1}c_{ij})^{n_{ij}}=1 , \ c_{is_{i}%
}=e_{i}^{-1}c_{i0}e_{i}$ \ (for $1 \leq i \leq k, 0 \leq j \leq s_{i} ), $

\qquad$x_{1} \ldots x_{r}e_{1} \ldots e_{k}a_{1}b_{1}a_{1}^{-1}b_{1}^{-1}
\ldots a_{h}b_{h}a_{h}^{-1}b_{h}^{-1}=1. $

If the sign is $-$ then we just replace the hyperbolic generators $a_{i}%
,b_{i}$ by glide reflections $d_{1},\ldots,d_{h}$, and the last relation by
$x_{1}\ldots x_{r}e_{1}\ldots e_{k}d_{1}^{2}\ldots d_{h}^{2}=1$.

The hyperbolic area of an arbitrary fundamental region of an NEC group
$\Gamma$ with signature (\ref{signature}) is given by%
\begin{equation}
\mu(\Gamma)=2\pi\left(  \varepsilon h-2+k+\sum_{i=1}^{r}{\left(  1-{\frac
{1}{m{_{i}}}}\right)  }+{\frac{1}{2}}\sum_{i=1}^{k}\sum_{j=1}^{s_{i}}{\left(
1-{\frac{1}{n{_{ij}}}}\right)  }\right)  \tag{2}%
\end{equation}
where $\varepsilon=2$ if the sign is $+$, and $\varepsilon=1$ if the sign is
$-$. Furthermore, any discrete group $\Lambda$ of isometries of $\mathbb{D}$
containing $\Gamma$ as a subgroup of finite index is also an NEC group, and
the hyperbolic area of a fundamental region for $\Lambda$ is given by the
Riemann-Hurwitz formula:%
\begin{equation}
\lbrack\Lambda:\Gamma]=\mu(\Gamma)/\mu(\Lambda). \tag{3}%
\end{equation}

The NEC groups with signature of the form $(h;+;[m_{1},\ldots,m_{r}];\{-\})$
are Fuchsian groups. For any NEC group $\Lambda$, let ${\Lambda}^{+}$ denote
the subgroup of orientation-preserving elements of $\Lambda$, called the
\textit{canonical Fuchsian subgroup\/} of $\Lambda$. If $\Lambda^{+}%
\neq\Lambda$ then $\Lambda^{+}$ has index $2$ in $\Lambda$ and we say that
$\Lambda$ is a \textit{proper\/} NEC group (see \cite{BEGG}).

\subsection{Riemann surfaces, automorphisms and uniformization groups}

A Riemann surface is a surface endowed with a complex analytical structure.
Let $X$ be a compact Riemann surface of genus $g>1$. Then there is a surface
Fuchsian group $\Gamma$ (that is, an NEC group with signature
$(g;+;[-];\{-\}))$) such that $X=\mathbb{D}/\Gamma$, and if $G$ is a group of
automorphisms of $X$ there is a Fuchsian group $\Delta$, containing $\Gamma$,
and an epimorphism $\theta:\Delta\rightarrow G$ such that $\ker\theta=\Gamma$.
If $G^{\ast}$ is a group of conformal and anticonformal automorphism then
there is an NEC\ group $\Lambda$, and an epimorphism $\theta^{\ast}%
:\Lambda\rightarrow G$ such that $\ker\theta^{\ast}=\Lambda$. In particular
the full automorphism group \textrm{Aut}$(X)$ of $X$ is isomorphic to
$\Delta/\Gamma$, where $\Delta$ is a Fuchsian group containing $\Gamma$. The
extended (full) automorphism group \textrm{Aut}$^{\pm}(X)$ of $X$ (including
anticonformal automorphisms) is isomorphic to $\Lambda/\Gamma$, where
$\Lambda$ is an NEC\ group such that $\Lambda^{+}=\Delta$.

\subsection{Topological types of anticonformal involutions}

Given a Riemann surface $X$ of genus $g$, the topological type of the action of
an anticonformal involution $\sigma\in\mathrm{Aut}(X)$ is determined by the
number of connected components, called \textit{ovals}, of its fixed point set
$Fix(\sigma)$ and the orientability of the\ Klein surface\ $X/\left\langle
\sigma\right\rangle $. We say that $\sigma$\ has \textit{species} $+k$ if
$Fix(\sigma)$ consists of $k$ ovals and $X/\left\langle \sigma\right\rangle $
is orientable, and $-k$ if $Fix(\sigma)$ consists of $k$ ovals and
$X/\left\langle \sigma\right\rangle $ is nonorientable (i. e. two surfaces
with symmetries of the same species have topologically conjugate quotient
orbifolds and vice versa). The set $Fix(\sigma)$ corresponds to the real part
of a complex algebraic curve representing $X$, which admits an equation with
real coefficients. The \textquotedblright$+$\textquotedblright\ sign in the
species of $\sigma$\ means that the real part disconnects its complement in
the complex curve and then we say that $\sigma$\ separates. By a classical
theorem of Harnack the possible values of species run between $-g$ and
$+(g+1)$, where $+k\equiv g+1\operatorname{mod}2$ (see \cite{CP} for a
geometrical proof).

A Riemann surface with an anticonformal involution is said to be a real
Riemann surface. The type of symmetry of a Riemann surface $X$ is the set of
topological types of anticonformal involutions of $X$.

There is a categorical equivalence between compact Riemann surfaces and
complex projective smooth algebraic curves. The conjugacy classes of
anticonformal involutions of Riemann surfaces correspond to the real forms of
the corresponding algebraic curve: i. e. real algebraic curves (see \cite{N}).
The topological type of an anticonformal involutions gives us important
information about the real points of a real algebraic curve, the number of
connected components of the real points of the algebraic curve and the
separability character of the real points inside the complex algebraic curve.

\subsection{Teichm\"{u}ller and moduli spaces}

Here we follow reference \cite{MS} on moduli spaces of Riemann and Klein surfaces.

Let $s$ be a signature of NEC groups and let $\mathcal{G}$ be an abstract
group isomorphic to the NEC groups with signature $s$. We denote by
$\mathbf{R}(s)$ the set of monomorphisms $r:\mathcal{G\rightarrow}Aut^{\pm
}(\mathbb{D})$ such that $r(\mathcal{G})$ is an NEC group with signature $s$.
The set $\mathbf{R}(s)$ has a natural topology given by the topology of
$Aut^{\pm}(\mathbb{D})$. Two elements $r_{1}$ and $r_{2}\in\mathbf{R}(s)$ are
said to be equivalent, $r_{1}\thicksim r_{2}$, if there exists $g\in Aut^{\pm
}(\mathbb{D})$ such that for each $\gamma\in\mathcal{G}$, $r_{1}%
(\gamma)=gr_{2}(\gamma)g^{-1}$. The space of classes $\mathbf{T}%
(s)=\mathbf{R}(s)/\thicksim$ is called the \textit{Teichm\"{u}ller space} of
NEC\ groups with signature $s$. If the signature $s$ is given in section
\ref{signature}, the Teichm\"{u}ller space $\mathbf{T}(s)$ is homeomorphic to
$\mathbb{R}^{d(s)}$, where
\[
d(s)=3(\varepsilon h-1+k)-3+(2r+\sum_{i=1}^{k}r_{i}).
\]

The modular group $\mathrm{Mod}(\mathcal{G})$ of $\mathcal{G}$ is the quotient
$\mathrm{Mod}(\mathcal{G})=\mathrm{Aut}(\mathcal{G})/\mathrm{Inn}%
(\mathcal{G})$, where $\mathrm{Inn}(\mathcal{G})$ denotes the inner
automorphisms of $\mathcal{G}$. The \emph{moduli space} of NEC groups with
signature $s$ is the quotient $\mathcal{M}_{s}=\mathbf{T}(s)/\mathrm{Mod}%
(\mathcal{G})$ endowed with the quotient topology. Hence $\mathcal{M}_{s}$ is
an orbifold with fundamental orbifold group $\mathrm{Mod}(\mathcal{G})$.

If $s$ is the signature of a surface group uniformizing surfaces of
topological type $t=(g,\pm,k)$, then we denote by $\mathbf{T}(s)=\mathbf{T}%
_{t}$ and $\mathcal{M}_{s}=\mathcal{M}_{t}$ the Teichm\"{u}ller and the moduli
space of Klein surfaces of topological type $t$.

Let $\mathcal{G}$ and $\mathcal{G}^{\prime}$ be abstract groups isomorphic to
NEC groups with signatures $s$ and $s^{\prime}$ respectively. Given an
inclusion mapping $\alpha:\mathcal{G}\rightarrow\mathcal{G}^{\prime}$ there is
an induced embedding $\mathbf{T}(\alpha):\mathbf{T}(s^{\prime})\rightarrow
\mathbf{T}(s)$ defined by $[r]\mapsto\lbrack r\circ\alpha]$.

If a finite group $G$ is isomorphic to a group of automorphisms of Klein
surfaces with topological type $t=(g,\pm,k)$, then the action of $G$ is
determined by an epimorphism $\theta:\mathcal{D}\rightarrow G$, where
$\mathcal{D}$ is an abstract group isomorphic to NEC groups with a given
signature $s$ and $\ker(\theta)=\mathcal{G}$ is a group isomorphic to NEC
surface groups uniformizing Klein surfaces of topological type $t$. Then there
is an inclusion $\alpha:\mathcal{G\rightarrow D}$ and an embedding
$\mathbf{T}(\alpha):\mathbf{T}(s)\rightarrow\mathbf{T}_{t}$. The continuous
map $\mathbf{T}(\alpha)$ induces a continuous map $\mathcal{M}_{s}%
\rightarrow\mathcal{M}_{t}$ and as a consequence:

\begin{proposition}
\label{Conected}\cite{MS} The set $\mathcal{B}_{t}^{G,\theta}$ of points in
$\mathcal{M}_{t}$ corresponding to surfaces having a group of automorphisms
isomorphic to $G$, with action determined by $\theta$, is a connected set.
\end{proposition}

\subsection{Compactification of moduli spaces}

A \textit{Riemann surface with nodes} is a connected complex analytic space
$S$ if and only if (see \cite{Be}):

\begin{enumerate}
\item there are $k=k(S)\geq0$ points $p_{1},...,p_{k}\in S$ called nodes such
that every node $p_{j}$ has a neighborhood isomorphic to the analytic set
$\{z_{1}z_{2}=0:\left\Vert z_{1}\right\Vert <1,\;\left\Vert z_{2}\right\Vert
<1\}$ with $p_{j}$ corresponding to $(0,0)$.

\item the set $S\smallsetminus\{p_{1},...,p_{k}\}$ has $r\geq1$ connected
components $\Sigma_{1},...,\Sigma_{r}$ called components of $S$, each of them
is a Riemann surface of genus $g_{i}$, with $n_{i}$ punctures with
$3g_{i}-3+n_{i}\geq0$ and $n_{1}+...+n_{r}=2k$.

\item we denote $g=(g_{1}-1)+...+(g_{r}-1)+k+1$
\end{enumerate}

If $k=k(S)=0$, $S$ is called \textit{non singular} and if $k=k(S)=3g-3$, $S$
is called \textit{terminal}.

To a Riemann surface with nodes $S$ we can associate a weighted graph,
\textit{the graph of} $S$, ${\mathcal{G}}(S)=(V_{S},E_{S},w)$, where $V_{S}$
is the set of vertices, $E_{S}$ is the set of edges, and $w$ is a function on
the set $V_{S}$ with non-negative integer values. This triple is defined in
the following way:

\begin{enumerate}
\item To each component $\Sigma_{i}$ corresponds a vertex in $V_{S}$.

\item To each node joining the components $\Sigma_{i}$ and $\Sigma_{j}$
corresponds and edge in $E_{S}$ connecting the corresponding vertices.
Multiple edges between the same pair of vertices and loops are allowed in
${\mathcal{G}}(S)$.

\item The function $w:V({\mathcal{G}}(S))\rightarrow{\mathbb{Z}}_{\geq0}$
associates to any vertex of ${\mathcal{G}}(S)$ the genus $g_{i}$ of
$\Sigma_{i}$.
\end{enumerate}

Let $\mathcal{M}_{g}$ be the moduli space of Riemann surfaces of genus $g$. A
well known result of Deligne and Mumford states that the set
$\widehat{{\mathcal{M}}_{g}}$ of Riemann surfaces with nodes of genus $g$ can
be endowed with a structure of projective complex variety and contains
$\mathcal{M}_{g}$ as a dense open subvariety \cite{DM}. If $g\geq2$ then
$\widehat{{\mathcal{M}}_{g}}$ is an irreducible complex projective variety of
dimension $3g-3$.

\section{Riemann surfaces of genus $g$ with $4g$ automorphisms}

\begin{lemma}
\label{Lemma Primero} Let $X$ be a Riemann surface of genus $g$ and let
$\Gamma$ be a surface Fuchsian group of genus $g$ uniformizing $X$. If $G$ is
an automorphism group of $X$, then $G\cong\Gamma^{\prime}/\Gamma$ where
$\Gamma^{\prime}$ is a Fuchsian group. If $\left\vert G\right\vert =4g$,
$g\neq3,6,15$ and $X$ is not in a finite set of exceptional Riemann surfaces
whose genera are $3,6,12$ or $30$, then the signature of $\Gamma^{\prime}$
must be:

\begin{enumerate}
\item $(0;+;[2,4g,4g])$

\item $(0;+;[3,6,2g])$

\item $(0;+;[4,4,2g])$

\item $(0;+;[2,2,2,2g])$
\end{enumerate}
\end{lemma}

\begin{proof}
Let $\Gamma^{\prime}$ have signature:
\[
(g^{\prime};+;[m_{1},...,m_{r}])
\]
By Riemann-Hurwitz formula we have:%
\[
\frac{2g-2}{2g^{\prime}-2+%
{\textstyle\sum\nolimits_{i=1}^{r}}
(1-\frac{1}{m_{i}})}=4g
\]
then%
\begin{equation}
2g^{\prime}-2+%
{\textstyle\sum\nolimits_{i=1}^{r}}
(1-\frac{1}{m_{i}})=\frac{1}{2}-\frac{1}{2g} \tag{4}%
\end{equation}
where we may assume that $m_{i-1}\leq m_{i}$, $i=2,...,r$. It is important to
note that $m_{i}$ divides $4g$ ($m_{i}$ is the order of a cyclic subgroup of
$G$). Hence $g^{\prime}=0$ and $r\leq4$, and formula (4) becomes:%
\[%
{\textstyle\sum\nolimits_{i=1}^{r}}
(1-\frac{1}{m_{i}})=\frac{5}{2}-\frac{1}{2g},r\leq4
\]

For $r=4$ we have $%
{\textstyle\sum\nolimits_{i=1}^{4}}
\frac{1}{m_{i}}=\frac{3}{2}+\frac{1}{2g}$, then if $g\neq3,6,15$ we have only
a solution $m_{1}=m_{2}=m_{3}=2,m_{4}=2g$, that is case 4 (note that for
$g=3,6,15$ we have the solutions $(2,2,3,3)$, $(2,2,3,4)$ and $(2,2,3,5)$ respectively).

If $r=3$ we have
\begin{equation}
\frac{1}{m_{1}}+\frac{1}{m_{2}}+\frac{1}{m_{3}}=\frac{1}{2}+\frac{1}{2g}
\tag{5}%
\end{equation}

From the formula (5) we have that $m_{1}\leq5$. If $m_{1}=2$, using the
formula and that $m_{i}$ divides $4g$ we have a unique solution $m_{2}%
=m_{3}=4g$ (case 1).

For $m_{1}=3,4,5$ it is possible to make a case by case analysis giving for
each value a bound for $m_{2}$ and for each possible value of $m_{2}$ a finite
set of solutions if $m_{3}\neq2g$. The solutions with $m_{3}\neq2g$ correspond
to following set of values of $g$:%
\begin{align*}
&  \{3,6,9,10,12,14,15,18,20,21,24,28,30,33,36,40,42,\\
&  45,60,66,72,84,90,105,126,132,153,190,273,276,420,429,861\}
\end{align*}

Using finite group theory and the algebra symbolic package MAGMA one shows
that there exist one or two exceptional surfaces exactly for genera $g=3,6,12$
or $30$. We thank Professor Marston Conder for helping us with these
calculations with MAGMA.

For $m_{3}=2g$ there are only two infinite set of solutions:%
\[
m_{1}=3,m_{2}=6,m_{3}=2g\text{ and }m_{1}=4,m_{2}=4,m_{3}=2g
\]
that are cases 2 and 3.
\end{proof}

\begin{remark}
See \cite{K}, section 2.3, for related results.
\end{remark}

\bigskip

In the next proposition we shall eliminate the cases 1, 2 and 3 of the
preceding Lemma using group theory and the fact that the order of
$\mathrm{Aut}(X)$ is exactly $4g$.

\bigskip

\begin{proposition}
\label{Proposition123} Let $X$ be a Riemann surface of genus $g$, uniformized
by a surface Fuchsian group $\Gamma$ and with full automorphism group
$\mathrm{Aut}(X)=G$ of order $4g$. If $\Gamma^{\prime}$ is a Fuchsian group
such that $\Gamma\leq\Gamma^{\prime}$ and $X/Aut(X)=\mathbb{D}/\Gamma^{\prime
}$ then the signature of $\Gamma^{\prime}$ is different from

\begin{enumerate}
\item $(0;+;[2,4g,4g])$

\item $(0;+;[3,6,2g])$

\item $(0;+;[4,4,2g])$
\end{enumerate}
\end{proposition}

\begin{proof}
\textbf{Case 1.} Assume that the signature of $\Gamma^{\prime}$ is
$(0;+;[2,4g,4g])$. Then there is a natural epimorphism $\theta:\Gamma^{\prime
}\rightarrow G\cong\Gamma^{\prime}/\Gamma$. If $\Gamma^{\prime}$ has a
canonical presentation $\left\langle x_{1},x_{2},x_{3}:x_{1}^{2}=x_{2}%
^{4g}=x_{3}^{4g}=x_{1}x_{2}x_{3}=1\right\rangle $ thus $\theta(x_{2})$ and
$\theta(x_{3})$ have order $4g$, since $\Gamma$ is a surface Fuchsian group.
Then $G$ is a cyclic group generated by $\theta(x_{3})=C$. We have
$\theta(x_{1})=C^{2g},\theta(x_{2})=C^{2g-1},\theta(x_{3})=C$.

The group $\Gamma^{\prime}$ is included in a Fuchsian group $\Delta$ of
signature $(0;+;[2,4,4g])$ (see \cite{S}). Let
\[
\left\langle x_{1}^{\prime},x_{2}^{\prime},x_{3}^{\prime}:x_{1}^{\prime
2}=x_{2}^{\prime4}=x_{3}^{\prime4g}=x_{1}^{\prime}x_{2}^{\prime}x_{3}^{\prime
}=1\right\rangle
\]
be a canonical presentation of $\Delta$. We have $x_{1}=x_{2}^{\prime2}$,
$x_{2}=x_{2}^{\prime-1}x_{3}^{\prime}x_{2}^{\prime}$, $x_{3}=x_{3}^{\prime}$
and an epimorphism $\theta^{\prime}:\Delta\rightarrow G^{\prime}$, where%
\[
G^{\prime}=\left\langle B,C:B^{2}=C^{2g},C^{4g}=1,B^{-1}CB=C^{2g-1}%
\right\rangle ,
\]
$\theta^{\prime}$ is defined by $\theta^{\prime}(x_{1}^{\prime})=C^{-1}%
B^{-1},\theta^{\prime}(x_{2}^{\prime})=B,\theta^{\prime}(x_{3}^{\prime})=C$.
Now $\theta^{\prime}\mid_{\Delta}=\theta$ and then the automorphism group of
$X$ has order $>4g$.

\textbf{Case 2.} Assume that the signature of $\Gamma^{\prime}$ is
$(0;+;[3,6,2g])$. Then there is a natural epimorphism: $\theta:\Gamma^{\prime
}\rightarrow G\cong\Gamma^{\prime}/\Gamma$. If
\[
\left\langle x_{1},x_{2},x_{3}:x_{1}^{3}=x_{2}^{6}=x_{3}^{2g}=x_{1}x_{2}%
x_{3}=1\right\rangle
\]
is a canonical presentation of $\Gamma^{\prime}$ then $G$ has a presentation
with generators $\theta(x_{1})=A,\theta(x_{2})=B,\theta(x_{3})=C$ and some of
the relations are:
\[
A^{3}=B^{6}=C^{2g}=ABC=1
\]
Hence $G$ is generated by $A$ and $C$.

Since $\Gamma^{\prime}$ is a surface group the order of $C$ is $2g$, then
$\left\langle C\right\rangle $ is an index two subgroup of $G$ and
$A\notin\left\langle C\right\rangle $. Hence $A^{2}\in\left\langle
C\right\rangle $, so $A^{2}=C^{t}$, and then $A=(A^{2})^{-1}=C^{2g-t}$, in
contradiction with $A\notin\left\langle C\right\rangle $.
\end{proof}

For the Case 3 we need a Lemma:

\begin{lemma}
Let $\Delta$ be a Fuchsian group with signature $(0;+;[4,4,2g])$ and let
\[
\left\langle x_{1},x_{2},x_{3}:x_{1}^{4}=x_{2}^{4}=x_{3}^{2g}=x_{1}x_{2}%
x_{3}=1\right\rangle
\]
be a canonical presentation of $\Delta$. Let $\theta:\Delta\rightarrow
G=\left\langle A,B\right\rangle $ be an epimorphism with kernel a surface
Fuchsian group and $\theta(x_{1})=A$, $\theta(x_{2})=B$.

There is a Fuchsian group $\Delta^{\prime}$ of signature $(0;+;[2,4,4g])$ with
$\Delta\leq\Delta^{\prime}$, $[\Delta:\Delta^{\prime}]=2,$ a group $G^{\prime
}$ with $G\leq G^{\prime}$, $[G:G^{\prime}]=2$, and an epimorphism
$\theta^{\prime}:\Delta^{\prime}\rightarrow G^{\prime}$, such that
$\theta^{\prime}\mid_{\Delta}=\theta$ if and only if the group $G$ admits an
automorphism $\alpha$ such that $\alpha(A)=B$, $\alpha(B)=A$.
\end{lemma}

\begin{proof}
If $G$ admits such an automorphism $\alpha$, then we can construct the
semidirect product $G^{\prime}=G\rtimes_{\alpha}C_{2}$, which is generated by
$G=\langle A,B\rangle$ and an order two element $D$, conjugation by which
induces the automorphism $\alpha$ on $G$. The group $\Delta$ is contained in
an NEC group $\Delta^{\prime}$ with signature $(0;+;[2,4,4g])$ and having
canonical generators $x_{1}^{\prime},x_{2}^{\prime},x_{3}^{\prime}$. Define an
epimorphism $\theta^{\prime}:\Delta^{\prime}\rightarrow G^{\prime}%
=G\rtimes_{\alpha}C_{2}$ by setting
\[
\theta^{\prime}(x_{1}^{\prime})=D,\quad\theta^{\prime}(x_{2}^{\prime}%
)=B,\quad\theta^{\prime}(x_{3}^{\prime})=DA^{-1}.
\]
Note that $G^{\prime}$ is isomorphic to $C_{4g}\rtimes C_{2}=\left\langle
DA^{-1}\right\rangle \rtimes\left\langle D\right\rangle $ and to
$C_{4g}\rtimes C_{4}=\left\langle DA^{-1}\right\rangle \rtimes\left\langle
B\right\rangle $.

Conversely, if such an extension $\theta^{\prime}:\Delta^{\prime}\rightarrow
G^{\prime}$ of $\theta$ exists and $\Delta^{\prime}$ has canonical generators
$x_{1}^{\prime}$,$x_{2}^{\prime},x_{3}^{\prime}$, then the embedding of
$\Delta$ in $\Delta^{\prime}$ is given by
\[
x_{1}\mapsto x_{1}^{\prime}x_{2}^{\prime}x_{1}^{\prime},\quad x_{2}\mapsto
x_{2}^{\prime},\quad x_{3}\mapsto x_{3}^{\prime2};
\]
hence if $D$ is the order two element $\theta^{\prime}(x_{1}^{\prime})$, then
\[
DAD=\theta^{\prime}(x_{1}^{\prime}x_{1}x_{1}^{\prime})=\theta^{\prime}%
(x_{2}^{\prime})=\theta(x_{2})=B
\]
and
\[
DBD=\theta^{\prime}(x_{1}^{\prime}x_{2}x_{1}^{\prime})=\theta^{\prime}%
(x_{1}^{\prime}x_{2}^{\prime}x_{1}^{\prime})=\theta(x_{1})=A,
\]
so conjugation by $D$ gives the required automorphism.\bigskip
\end{proof}

Now we continue the proof of the Proposition:

\begin{proof}
\textbf{Case 3}. Assume that the signature of $\Gamma^{\prime}$ is
$(0;+;[4,4,2g])$. Then there is a natural epimorphism $\theta:\Gamma^{\prime
}\rightarrow G\cong\Gamma^{\prime}/\Gamma$. If
\[
\left\langle x_{1},x_{2},x_{3}:x_{1}^{4}=x_{2}^{4}=x_{3}^{2g}=x_{1}x_{2}%
x_{3}=1\right\rangle
\]
is a canonical presentation of $\Gamma^{\prime}$ then $G$ has a presentation
with generators $\theta(x_{1})=A,\theta(x_{2})=B,\theta(x_{3})=C$ and some of
the relations are $A^{4}=B^{4}=C^{2g}=ABC=1$. Hence $G$ is generated by $A$
and $C$.

Since the order of $\left\langle C\right\rangle $ is $2g$ then $A^{2}%
\in\left\langle C\right\rangle $ and since $\left\langle C\right\rangle
\vartriangleleft G$ then $ACA^{-1}=C^{t}$. As $\Gamma$ is a surface Fuchsian
group, $A^{2}$ has order two and $A^{2}=C^{g}$. We have that $A^{-1}C^{-1}$
has order four, then:%
\[
(A^{-1}C^{-1})^{4}=1,ACA^{-1}=C^{t},A^{2}=C^{g}%
\]
From the above relations we have that $2(t+1)\equiv0\operatorname{mod}2g$,
then either $t=g-1$ or $t=2g-1$.

If $t=g-1$, then $A^{-1}C^{-1}$ has order two but as $\Gamma$ is a surface
Fuchsian group, then $A^{-1}C^{-1}$ must have order four, so this case is not possible.

If $t=2g-1$ we have the relation $ACA^{-1}=C^{-1}$. The group $G$ has
presentation:%
\[
\left\langle A,C:A^{4}=C^{2g}=1;ACA^{-1}=C^{-1};A^{2}=C^{g}\right\rangle
\]
The assignation $A\rightarrow A^{-1}C^{-1}$ and $C\rightarrow C^{-1}$ defines
an automorphism such that $A\rightarrow A^{-1}C^{-1}=B$ and $B=A^{-1}%
C^{-1}\rightarrow A$. By the preceding Lemma the automorphism group contains
properly $G$ and then $\left\vert Aut(X)\right\vert >4$.
\end{proof}

\begin{remark}
For all $g\geq2$, there is a Riemann surface $X_{8g}=\,\mathbb{D}/\Gamma$, the
Wiman curve of type II: $w^{2}=z(z^{2g}-1)$ (see \cite{W}), with $8g$
automorphisms (except for $g=2$) and such that $X_{8g}/\mathrm{Aut}(X_{8g})$
is uniformized by a group of signature $(0;+;[2,4,4g])$ containing $\Gamma$.
The groups $G^{\prime}$ in cases 1 and 3 are isomorphic to $\mathrm{Aut}%
(X_{8g})$. The full automorphism group of Wiman's curve of genus $2$ is
$GL(2.3)$, of order $48$.
\end{remark}

\begin{theorem}
\label{Theorem 1}Let $X$ be a Riemann surface of genus $g$ uniformized by a
surface Fuchsian group $\Gamma$ and with (full) automorphism group $G$ of
order $4g$. Assume that $g\neq3,6,15$ and $X$ is not in the finite set of
exceptional Riemann surfaces in Lemma \ref{Lemma Primero}. If $\Gamma^{\prime
}$ is a Fuchsian group such that $\Gamma\leq\Gamma^{\prime}$ and
$X/Aut(X)=\mathbb{D}/\Gamma^{\prime}$ then the signature of $\Gamma^{\prime}$
is $(0;+;[2,2,2,2g])$ and $G\cong D_{2g}$ (the dihedral group of $4g$ elements).
\end{theorem}

\begin{proof}
Let $X$ be a Riemann surface of genus $g$, uniformized by a surface Fuchsian
group $\Gamma$ and with automorphism group $G$ of order $4g$. If
$\Gamma^{\prime}$ is a Fuchsian group with $\Gamma\leq\Gamma^{\prime}$ and
$X/Aut(X)=\mathbb{D}/\Gamma^{\prime}$ then, by Lemma \ref{Lemma Primero} and
Proposition \ref{Proposition123} the signature of $\Gamma^{\prime}$ is
$(0;+;[2,2,2,2g])$.

There is a canonical presentation of $\Gamma^{\prime}$:%
\[
\left\langle x_{1},x_{2},x_{3},x_{4}:x_{i}^{2}=x_{4}^{2g}=x_{1}x_{2}x_{3}%
x_{4}=1,i=1,2,3\right\rangle
\]
and an epimorphism:%
\[
\theta:\Gamma^{\prime}\rightarrow G\cong\Gamma^{\prime}/\Gamma
\]
If $\theta(x_{4})=D$, we have that the order of $D$ is $2g$. Some of the
$\theta(x_{i})$, $i=1,2,3$, does not belong to $\left\langle D\right\rangle $,
using, if necessary, an automorphism of $\Gamma^{\prime}$ we may suppose that
is $\theta(x_{1})=A\notin\left\langle D\right\rangle $.\ Then $A^{2}=1$ and
since $\left\langle D\right\rangle \vartriangleleft G$, $ADA^{-1}=D^{t}$, with
$t^{2}\equiv1\operatorname{mod}2g$.

The elements $\theta(x_{2})$ and $\theta(x_{3})$ have order $2$ and all order
two elements in $G=\left\langle A,D\right\rangle $ are $A$, $D^{g}$ and
$D^{r}A$ with $r(t+1)\equiv0\operatorname{mod}2g$. Since $x_{1}x_{2}x_{3}%
x_{4}=1$, $\theta(x_{2}x_{3})=A^{-1}D^{-1}$, therefore either $\theta
(x_{2})=D^{g}$ and $\theta(x_{3})=D^{r}A$ or $\theta(x_{2})=D^{r}A$ and
$\theta(x_{3})=D^{g}$. Using if necessary an automorphism of $\Gamma^{\prime}$
we may assume $\theta(x_{2})=D^{r}A$ and $\theta(x_{3})=D^{g}$. Finally using
$\theta(x_{2}x_{3})=A^{-1}D^{-1}$ we obtain $D^{r}AD^{g}=A^{-1}D^{-1}$ from
where we have $rt+g+1\equiv0\operatorname{mod}(2g)$. As $r(t+1)\equiv
0\operatorname{mod}2g$, we have $g+1-r\equiv0\operatorname{mod}(2g)$ and
$r\equiv g+1\operatorname{mod}2g$, then $r=g+1$ and $t=-1$. Hence $ADA=D^{-1}%
$, $A^{2}=D^{2g}=1$, the group $G$ is $D_{2g}$ and the epimorphism is unique
(up to automorphisms of $\Gamma^{\prime}$ and $G$):%
\[
\theta(x_{1})=A,\theta(x_{2})=D^{g+1}A,\theta(x_{3})=D^{g},\theta(x_{4})=D
\]

\end{proof}

\begin{remark}
Note that the epimorphism $\theta:\Gamma^{\prime}\rightarrow G\cong%
\Gamma^{\prime}/\Gamma$ of the Theorem \ref{Theorem 1} is unique up to
automorphisms of $\Gamma^{\prime}$ and $G$. So the surfaces of genus $g$
having automorphism group of order $4g$ with $g\geq31$ or in the conditions of
the Theorem, form a connected equisymmetric uniparametric family.
\end{remark}

\begin{remark}
The surfaces in the above theorem are hyperelliptic. The hyperelliptic
involution corresponds to the element $D^{g}$ of $D_{2g}$, since $\theta
^{-1}(D)$ has signature $(0;+;[2,\overset{2g+2}{...},2])$ (see \cite{BEGG}).
\end{remark}

\section{Conformal and anticonformal automorphism groups}

In this section we shall obtain the groups of conformal and anticonformal
automorphisms of Riemann surfaces with automorphism group of order $4g$.

\begin{theorem}
\label{Signaturas NEC}Let $X$ be a Riemann surface of genus $g$, uniformized
by a surface Fuchsian group $\Gamma$ and with automorphism group $G$ of order
$4g$. If $\Gamma^{\prime}$ is a Fuchsian group with $\Gamma\leq\Gamma^{\prime
}$ and $X/Aut(X)=\mathbb{D}/\Gamma^{\prime}$, we assume that the signature of
$\Gamma^{\prime}$ is $(0;+;[2,2,2,2g])$. Let $Aut^{\pm}(X)=G^{\ast}$ be the
group of conformal and anticonformal automorphisms of $X$ and $\Gamma^{\ast}$
be an NEC group such that $G^{\ast}\cong\Gamma^{\ast}/\Gamma$. If $G^{\ast
}\gneq G$ then the signature of $\Gamma^{\ast}$ is

\begin{enumerate}
\item[a.] $(0;+;[-];(2,2,2,2g))$ then $G^{\ast}\cong D_{2g}\times C_{2}$ and
there are two epimorphisms $\Gamma^{\ast}\rightarrow D_{2g}\times C_{2}$ (up
to automorphisms of $\Gamma^{\ast}$).

\item[b.] $(0;+;[2];(2,2g))$ then $G^{\ast}$ has presentation:%
\begin{align*}
&  \left\langle x,z,w:x^{2}=z^{2}=w^{2}=(zw)^{2g}=1,xzx=(zw)^{g-1}%
z,xwx=(zw)^{g}z\right\rangle \\
&  =D_{2g}\rtimes_{\varphi}C_{2}%
\end{align*}
where $\varphi(z)=(zw)^{g-1}z$, $\varphi(w)=(zw)^{g}z$. Then $G^{\ast}\cong
D_{4g}$ if $g$ is even and $G^{\ast}\cong D_{2g}\times C_{2}$ if $g$ is odd.
\end{enumerate}
\end{theorem}

\begin{proof}
Since the signature of $\Gamma^{\prime}$ is $(0;+;[2,2,2,2g])$ and
$\Gamma^{\prime}$ is an index two subgroup of the NEC group $\Gamma^{\ast}$,
the signature of $\Gamma^{\ast}$ must be either:

a. $(0;+;[-];(2,2,2,2g))$ or

b. $(0;+;[2];(2,2g)).$

Case a. $\Gamma^{\ast}$ has signature $(0;+;[-];(2,2,2,2g))$. Let
\[
\left\langle c_{0},c_{1},c_{2},c_{3}:c_{i}^{2}=(c_{0}c_{1})^{2}=(c_{1}%
c_{2})^{2}=(c_{2}c_{3})^{2}=(c_{3}c_{0})^{2g}=1\right\rangle
\]
be a canonical presentation of $\Gamma^{\ast}$. Assume that the epimorphism
$\theta^{\ast}:\Gamma^{\ast}\rightarrow G^{\ast}\cong\Gamma^{\ast}/\Gamma$, is
given by%
\[
\theta^{\ast}(c_{0})=x,\theta^{\ast}(c_{1})=y,\theta^{\ast}(c_{2}%
)=z,\theta^{\ast}(c_{3})=w,
\]
then we have that $\{x,y,z,w\}$ is a set of generators of $G^{\ast}$ and,
since $\Gamma$ is a surface Fuchsian group, $x,y,z,w,xy,yz,zw\,$have order 2
and $wx$ has order $2g$. Since $\left\vert G^{\ast}\right\vert =8g$ and
$\left\langle x,w\right\rangle $ has order $4g$ then $\left\langle
x,w\right\rangle \vartriangleleft G^{\ast}$. Note that either $y$ or $z$ is
not in $\left\langle x,w\right\rangle $, assume that $y\notin\left\langle
x,w\right\rangle $ (the argument assuming $z\notin\left\langle
x,w\right\rangle $ is annalogous). Hence $y(wx)y=(wx)^{t}$, with $(t,2g)=1$.

The elements $y$ and $z$ are not the same, because $yz$ has order 2. Since
$G^{\ast}=G^{\ast}\cup yG^{\ast}$ we have two possibilities either $z\in
G^{\ast}$ or $z\in yG^{\ast}$.

Case 1. $z\in G^{\ast}$, then either $z=(wx)^{g}$ or $z=(wx)^{g}w$.

Case 1a. The equality $z=(wx)^{g}$ is not possible, since $(wx)^{g}$ is an
orientation preserving element.

Case 1b. If $z=(wx)^{g}w$, since $(yz)^{2}=1$ we have $y(wx)^{g}%
wy(wx)^{g}w=(wx)^{gt(t+1)}ywyw=(yw)^{2}=1$. Then $(xy)^{2}=(yz)^{2}=(yw)^{2}$,
and $G^{\ast}=D_{2g}\times C_{2}=\left\langle x,w\right\rangle \times
\left\langle y\right\rangle $. The epimorphism $\theta^{\ast}:\Gamma^{\ast
}\rightarrow G^{\ast}\cong\Gamma^{\ast}/\Gamma$, completely determined up to
automorphisms of $\Gamma^{\ast}$ or $G^{\ast}$, is:%
\[
\theta^{\ast}(c_{0})=x,\theta^{\ast}(c_{1})=y,\theta^{\ast}(c_{2}%
)=(wx)^{g}w,\theta^{\ast}(c_{3})=w
\]

Note that $\Gamma^{\prime}$ is the canonical Fuchsian subgroup of
$\Gamma^{\ast}$. We shall see that the epimorphism $\theta^{\ast}$ restricted
to $\Gamma^{\prime}$ is equivalent by automorphisms of $\Gamma^{\prime}$ and
$D_{2g}$ to the epimorphism constructed in the proof of Theorem
\ref{Theorem 1}. A set of generators of a canonical presentation of
$\Gamma^{\prime}$ expresed in terms of the canonical presentation of
$\Gamma^{\ast}$ is:%
\[
\{x_{1}^{\prime}=c_{0}c_{1},x_{2}^{\prime}=c_{1}c_{2},x_{3}^{\prime}%
=c_{2}c_{3},x_{4}^{\prime}=c_{3}c_{0}\}
\]

The restriction of $\theta^{\ast}$ is:%
\begin{align*}
x_{1}^{\prime}  &  \rightarrow xy,x_{2}^{\prime}\rightarrow y(wx)^{g}w\\
x_{3}^{\prime}  &  \rightarrow(wx)^{g},x_{4}^{\prime}\rightarrow wx
\end{align*}
and $\left\langle xy,wx\right\rangle \cong D_{2g}$ since $xy(wx)(xy)^{-1}%
=xw=(wx)^{-1}$. Hence $\theta^{\ast}$ restricted to $\Gamma^{\prime}$ is
exactly the epimorphism in the proof of Theorem \ref{Theorem 1}, where $xy=A$
and $D=wx$. Note that $\theta^{\ast}(\Gamma^{\prime})=\left\langle
xy,wx\right\rangle \cong D_{2g}$, is not the subgroup $\left\langle
z,w\right\rangle \cong D_{2g}$ used in the construction of $G^{\ast}$.

Case 2. If $z\in yG^{\ast}$ then either $z=y(wx)^{s}$ or $z=y(wx)^{s}w$. Since
$y(wx)^{s}w$ is orientation preserving, the second case is not possible.
Assume $z=y(wx)^{s}$. From $z^{2}=1$ we have:%
\[
y(wx)^{s}y(wx)^{s}=(wx)^{st+s}=1
\]
so $s(t+1)\equiv0\operatorname{mod}2g$.

We have $(yz)^{2}=1$ then%
\[
yy(wx)^{s}yy(wx)^{s}=(wx)^{2s}=1
\]
so $s=g$ and $g(t+1)\equiv0\operatorname{mod}2g$.

Finally we have $(zw)^{2}=1$ then%
\[
y(wx)^{g}wy(wx)^{g}w=1
\]%
\[
(wx)^{tg}ywy(wx)^{g}w=1
\]%
\[
ywy=(wx)^{g(t-1)}w
\]
and by $g(t+1)\equiv0\operatorname{mod}2g$ we have $ywy=w$, then $G^{\ast
}=D_{2g}\times C_{2}=\left\langle x,w\right\rangle \times\left\langle
y\right\rangle $ and
\[
\theta^{\ast}(c_{0})=x,\theta^{\ast}(c_{1})=y,\theta^{\ast}(c_{2}%
)=y(wx)^{g},\theta^{\ast}(c_{3})=w
\]
The epimorphism $\theta^{\ast}$ is unique up to automorphism of $\Gamma^{\ast
}$ or $G^{\ast}$.

As in the preceding case, we shall see that the epimorphism $\theta^{\ast}$,
restricted to $\Gamma^{\prime}$, is equivalent by automorphisms of
$\Gamma^{\prime}$ and $G^{\ast}$ to the epimorphism constructed in the proof
of Theorem \ref{Theorem 1}. As before, a set of generators of a canonical
presentation of $\Gamma^{\prime}$ expressed in terms of the canonical
presentation of $\Gamma^{\ast}$ is:%
\[
\{x_{1}^{\prime}=c_{0}c_{1},x_{2}^{\prime}=c_{1}c_{2},x_{3}^{\prime}%
=c_{2}c_{3},x_{4}^{\prime}=c_{3}c_{0}\}
\]

The restriction of $\theta^{\ast}$ is:%
\begin{align*}
x_{1}^{\prime}  &  \rightarrow xy,x_{2}^{\prime}\rightarrow(wx)^{g}\\
x_{3}^{\prime}  &  \rightarrow(wx)^{g}w,x_{4}^{\prime}\rightarrow wx
\end{align*}
and $\left\langle xy,wx\right\rangle \cong D_{2g}$. Hence $\theta^{\ast}$
restricted to $\Gamma^{\prime}$ is exactly the epimorphism in the proof of
Theorem \ref{Theorem 1}.

Case b. $\Gamma^{\ast}$ has signature $(0;+;[2];(2,2g))$. Let
\[
\left\langle a,c_{0},c_{1},c_{2}:a^{2}=c_{i}^{2}=(c_{0}c_{1})^{2}=(c_{1}%
c_{2})^{2g}=xc_{0}xc_{2}=1\right\rangle
\]
be a canonical presentation of $\Gamma^{\ast}$. Assume that the epimorphism
$\theta^{\ast}:\Gamma^{\ast}\rightarrow G^{\ast}\cong\Gamma^{\ast}/\Gamma$, is
given by%
\[
\theta^{\ast}(a)=x,\theta^{\ast}(c_{0})=y,\theta^{\ast}(c_{1})=z,\theta^{\ast
}(c_{2})=w.
\]
Then we have that $\{x,y,z,w\}$ is a set of generators of $G^{\ast}$ and
$x,y,z,w,yz\,$have order $2$, $zw$ has order $2g$ and $xyxw=1$.

As $xyxw=1$ and $G^{\ast}\supsetneq G$ we have that $x\notin\left\langle
z,w\right\rangle \cong D_{2g}$. The group $\left\langle z,w\right\rangle $ has
index two in $G^{\ast}$, then is a normal subgroup of $G^{\ast}$. Hence:
\[
xzx\in\left\langle z,w\right\rangle \text{ and }xwx\in\left\langle
z,w\right\rangle
\]

Also we have that $xzx$ and $xwx$ are the images by $\theta^{\ast}$ of
orientation reversing transformations, then:%
\[
xzx=(zw)^{t_{1}}z\text{ and }xwx=(zw)^{t_{2}}z
\]
Using that $(yz)^{2}=1$, we have
\[
(xwx)z(xwx)z=(zw)^{t_{2}}zz(zw)^{t_{2}}zz=(zw)^{2t_{2}}=1
\]
from where $t_{2}=g$. Again by $(yz)^{2}=1$, we have%
\[
(xwx)z(xwx)z=1\text{, then }w(xzx)w(xzx)=1
\]
so%
\[
w(zw)^{t_{1}}zw(zw)^{t_{1}}z=(zw)^{2t_{1}+2}=1,
\]
then $t_{1}=g-1$.

We have that the group $G^{\ast}$ has presentation:%
\begin{gather*}
\left\langle x,z,w:x^{2}=z^{2}=w^{2}=(zw)^{2g}=1,xzx=(zw)^{g-1}z,xwx=(zw)^{g}%
z\right\rangle \\
\cong D_{2g}\rtimes_{\varphi}C_{2}=\left\langle z,w\right\rangle
\rtimes_{\varphi}\left\langle x\right\rangle
\end{gather*}
where $\varphi:D_{2g}\rightarrow D_{2g}$ is $z\rightarrow(zw)^{g-1}z$ and
$w\rightarrow(zw)^{g}z$.

The epimorphism $\theta^{\ast}$, unique up to automorphisms of $\Gamma^{\ast}$
or $G^{\ast}$, is:%
\[
a\rightarrow x;c_{0}\rightarrow xwx;c_{1}\rightarrow z;c_{2}\rightarrow w
\]

Note that $\Gamma^{\prime}$ is the canonical Fuchsian subgroup of
$\Gamma^{\ast}$. We shall see that the epimorphism $\theta^{\ast}$ restricted
to $\Gamma^{\prime}$ is equivalent, by automorphisms of $\Gamma^{\prime}$ and
$D_{2g}$, to the epimorphism constructed in the proof of Theorem
\ref{Theorem 1}. A set of generators of a canonical presentation of
$\Gamma^{\prime}$ expressed in terms of the canonical presentation of
$\Gamma^{\ast}$ is:%
\[
\{x_{1}^{\prime}=a,x_{2}^{\prime}=c_{0}ac_{0},x_{3}^{\prime}=c_{0}c_{1}%
,x_{4}^{\prime}=c_{1}ac_{0}a=c_{1}c_{2}\}
\]

The restriction is:%
\begin{align*}
x_{1}^{\prime}  &  \rightarrow x,x_{2}^{\prime}\rightarrow xwxwx=(zw)^{g}%
zwx=(zw)^{g+1}x\\
x_{3}^{\prime}  &  \rightarrow xwxz=(zw)^{g},x_{4}^{\prime}\rightarrow zw
\end{align*}
and $\left\langle x,zw\right\rangle \cong D_{2g}$ since%
\[
xzwx=xzxxwx=(zw)^{g-1}z(zw)^{g}z=(zw)^{-1}%
\]
Hence $\theta^{\ast}$ restricted to $\Gamma^{\prime}$ is exactly the
epimorphism in the proof of Theorem \ref{Theorem 1}, setting $x=A$ and $D=zw$.
Note that $\theta^{\ast}(\Gamma^{\prime})$ is $\left\langle x,zw\right\rangle
\cong D_{2g}$ but it is not the subgroup $\left\langle z,w\right\rangle \cong
D_{2g}$ used in the construction of $G^{\ast}$.

Since $xzx=(zw)^{g-1}z$, then $(zx)^{2}=(zw)^{g-1}$. If $g$ is even then $zx$
has order $4g$ and $D_{2g}\rtimes_{\varphi}C_{2}$ is isomorphic to $D_{4g}$.
Finally if $g$ is odd then $(xz)^{g}$ has of order 2 and it is in the center
of $D_{2g}\rtimes_{\varphi}C_{2}$, then $D_{2g}\rtimes_{\varphi}C_{2}%
\cong\left\langle x,w\right\rangle \times\left\langle (xz)^{g}\right\rangle
\cong D_{2g}\times C_{2}$.
\end{proof}

\section{Symmetry type of Riemann surfaces with automorphism group of order
$4g$}

\begin{theorem}
Let $X$ be a Riemann surface of genus $g$, uniformized by a surface Fuchsian
group $\Gamma$ and with automorphism group $G$ of order $4g$. If
$\Gamma^{\prime}$ is a Fuchsian group with $\Gamma\leq\Gamma^{\prime}$ and
$X/Aut(X)=\mathbb{D}/\Gamma^{\prime}$, we assume that the signature of
$\Gamma^{\prime}$ is $(0;+;[2,2,2,2g])$. Let $\mathrm{Aut}^{\pm}(X)=G^{\ast}$
be the extended automorphism group of $X$ and let $\Gamma^{\ast}$ be an NEC
group such that $G^{\ast}\cong\Gamma^{\ast}/\Gamma$. Assume that the signature
of $\Gamma^{\ast}$ is $(0;+;[-];\{(2,2,2,2g)\})$. Then there are four
conjugacy classes of anticonformal involutions and the sets of topological
types are either $\{+2,0,-2,-2\}$ if $g$ is odd and $\{+1,0,-1,-3\}$ if $g$ is
even, or $\{-1$,$-1,-g,-g\}$.
\end{theorem}

\begin{proof}
By Theorem \ref{Signaturas NEC} the automorphism group in this case is
isomorphic to
\[
D_{2g}\times C_{2}=\left\langle w,x:w^{2}=x^{2}=(wx)^{2g}=1\right\rangle
\times\left\langle y:y^{2}=1\right\rangle .
\]

There are two possible epimorphisms $\theta_{i}^{\ast}:\Gamma^{\ast
}\rightarrow G^{\ast}$, $i=1,2$:%
\[
\theta_{1}^{\ast}(c_{0})=x,\theta_{1}^{\ast}(c_{1})=y,\theta_{1}^{\ast}%
(c_{2})=(wx)^{g}w,\theta_{1}^{\ast}(c_{3})=w
\]
and%
\[
\theta_{2}^{\ast}(c_{0})=x,\theta_{2}^{\ast}(c_{1})=y,\theta_{2}^{\ast}%
(c_{2})=y(wx)^{g},\theta_{2}^{\ast}(c_{3})=w.
\]

There are four conjugacy classes of involutions in $D_{2g}\times C_{2}$ not in
$\theta_{i}^{\ast}(\Gamma^{\prime})$, a set of representatives of each class
is $\{x,y,y(wx)^{g},w\}$.

For each involution $\iota$ in $\mathrm{Aut}^{\pm}(X)=G^{\ast}$ the number of
fixed ovals of $\iota$ is given by the following formula of G. Gromadzki (cf.
\cite{G}):%
\[%
{\textstyle\sum\nolimits_{c_{i}\text{, s.t. }\theta_{i}^{\ast}(c_{i}%
)^{h}=\iota}}
\left[  C(G^{\ast},\theta_{i}^{\ast}(c_{i})):\theta_{i}^{\ast}(C(\Gamma^{\ast
},c_{i}))\right]  .
\]

For the epimorphism $\theta_{1}^{\ast}$ we have the following centralizers:

$C(G^{\ast},\theta_{1}^{\ast}(c_{0}))=C(G^{\ast},x)=\left\langle
x,(wx)^{g}\right\rangle \times\left\langle y\right\rangle $ and $\theta
_{1}^{\ast}(C(\Gamma^{\ast},c_{0}))=\left\langle x,(wx)^{g}\right\rangle
\times\left\langle y\right\rangle $

$C(G^{\ast},\theta_{1}^{\ast}(c_{1}))=C(G^{\ast},y)=G^{\ast}$ and $\theta
_{1}^{\ast}(C(\Gamma^{\ast},c_{1}))=\left\langle x,(wx)^{g}w\right\rangle
\times\left\langle y\right\rangle $

$C(G^{\ast},\theta_{1}^{\ast}(c_{2}))=C(G^{\ast},(wx)^{g}w)=\left\langle
w,(wx)^{g}\right\rangle \times\left\langle y\right\rangle $ and $\theta
_{1}^{\ast}(C(\Gamma^{\ast},c_{3}))=\left\langle w,(wx)^{g}\right\rangle
\times\left\langle y\right\rangle $

$C(G^{\ast},\theta_{1}^{\ast}(c_{3}))=C(G^{\ast},w)=\left\langle
w,(wx)^{g}\right\rangle \times\left\langle y\right\rangle $ and $\theta
_{1}^{\ast}(C(\Gamma^{\ast},c_{3}))=\left\langle w,(wx)^{g}\right\rangle .$

For the class of involutions $[x]$ we have either:

$\left[  C(G^{\ast},\theta_{1}^{\ast}(c_{0})):\theta_{i}^{\ast}(C(\Gamma
^{\ast},c_{0}))\right]  =1$ oval, if $g$ is even or

$\left[  C(G^{\ast},\theta_{1}^{\ast}(c_{0})):\theta_{i}^{\ast}(C(\Gamma
^{\ast},c_{0}))\right]  +\left[  C(G^{\ast},\theta_{1}^{\ast}(c_{2}%
)):\theta_{i}^{\ast}(C(\Gamma^{\ast},c_{2}))\right]  =2$ ovals, if $g$ is odd.

Note that $\left\langle x,(wx)^{g}w\right\rangle $ is isomorphic to $D_{2g}$
if $g$ is even and it is isomorphic to $D_{g}$ if $g$ is odd. Hence the class
of involutions $[y]$ has $2$ ovals if $g$ is odd and $1$ oval if $g$ is even.

There is no reflection $c_{i}$ such that $\theta_{1}^{\ast}(c_{i})$ is in the
conjugacy class represented by $y(wx)^{g}$, then the anticonformal involutions
in the class $[y(wx)^{g}]$ have no ovals.

Finally for the class of involutions $[w]$ we have either:

$\left[  C(G^{\ast},\theta_{1}^{\ast}(c_{2})):\theta_{i}^{\ast}(C(\Gamma
^{\ast},c_{2}))\right]  +\left[  C(G^{\ast},\theta_{1}^{\ast}(c_{3}%
)):\theta_{i}^{\ast}(C(\Gamma^{\ast},c_{3}))\right]  =3$ ovals, if $g$ is even or

$\left[  C(G^{\ast},\theta_{1}^{\ast}(c_{3})):\theta_{i}^{\ast}(C(\Gamma
^{\ast},c_{3}))\right]  =2$ ovals, if $g$ is odd.

The set of topological types is $\{\pm2,\pm2,\pm2,0\}$ if $g$ is odd and
$\{\pm3,\pm1,\pm1,0\}$ if $g$ is even. Now applying Theorem 3.4.4 of
\cite{BCGG}, the topological types of the anticonformal involutions are
$\{+2,0,-2,-2\}$ if $g$ is odd and $\{+1,0,-1,-3\}$ if $g$ is even.\smallskip

For the epimorphism $\theta_{2}^{\ast}$ we have:

$C(G^{\ast},\theta_{2}^{\ast}(c_{0}))=C(G^{\ast},x)=\left\langle
x,(wx)^{g}\right\rangle \times\left\langle y\right\rangle $ and $\theta
_{1}^{\ast}(C(\Gamma^{\ast},c_{0}))=\left\langle x,(wx)^{g}\right\rangle
\times\left\langle y\right\rangle $, then $[x]$ has one oval.

$C(G^{\ast},\theta_{2}^{\ast}(c_{1}))=C(G^{\ast},y)=G^{\ast}$ and $\theta
_{1}^{\ast}(C(\Gamma^{\ast},c_{1}))=\left\langle x,(wx)^{g}\right\rangle
\times\left\langle y\right\rangle $, thus $[y]$ has $g$ ovals.

$C(G^{\ast},\theta_{2}^{\ast}(c_{2}))=C(G^{\ast},y(wx)^{g})=G^{\ast}$ and
$\theta_{1}^{\ast}(C(\Gamma^{\ast},c_{2}))=\left\langle w,(wx)^{g}%
\right\rangle \times\left\langle y\right\rangle $, then $[y(wx)^{g}]$ has $g$ ovals.

$C(G^{\ast},\theta_{2}^{\ast}(c_{3}))=C(G^{\ast},w)=\left\langle
w,(wx)^{g}\right\rangle \times\left\langle y\right\rangle $ and $\theta
_{1}^{\ast}(C(\Gamma^{\ast},c_{3}))=\left\langle w,(wx)^{g}\right\rangle
\times\left\langle y\right\rangle $, thus $[w]$ has one oval.

The set of topological types is $\{\pm1,\pm1,-g,-g\}$.

By Theorem 3.4.4 of \cite{BCGG} the topological types of the anticonformal
involutions are $\{-1,-1,-g,-g\}$.
\end{proof}

\begin{theorem}
Let $X$ be a Riemann surface of genus $g$, uniformized by a surface Fuchsian
group $\Gamma$ and with automorphism group $G$ of order $4g$. If
$\Gamma^{\prime}$ is a Fuchsian group with $\Gamma\leq\Gamma^{\prime}$ and
$X/Aut(X)=\mathbb{D}/\Gamma^{\prime}$, we assume that the signature of
$\Gamma^{\prime}$ is $(0;+;[2,2,2,2g])$. Let $Aut^{\pm}(X)=G^{\ast}$ be the
extended automorphism group of $X$ and $\Gamma^{\ast}$ be an NEC group such
that $G^{\ast}\cong\Gamma^{\ast}/\Gamma$. Assume that the signature of
$\Gamma^{\ast}$ is $(0;+;[2];\{(2,2g)\})$. The set of topological types of the
anticonformal involutions of $X$ is $\{0,0,-2,-2\}$ if the genus $g$ is odd
and $\{-2\}$ if the genus $g$ is even.
\end{theorem}

\begin{proof}
By Theorem \ref{Signaturas NEC} the automorphism group in this case is
isomorphic to
\begin{align*}
&  \left\langle x,z,w:x^{2}=z^{2}=w^{2}=(zw)^{2g}=1,xzx=(zw)^{g-1}%
z,xwx=(zw)^{g}z\right\rangle \\
&  =D_{2g}\rtimes_{\varphi}C_{2}%
\end{align*}
If $g$ is odd, there are four conjugacy classes of orientation reversing order
two elements in $Aut^{\pm}(X)\cong D_{2g}\times C_{2}=\left\langle
x,w\right\rangle \times\left\langle (xz)^{g}\right\rangle $, a set of
representatives of each class is $\{z,w,(xz)^{g},(xw)^{g}\}$. If $g$ is even
the group $D_{2g}\rtimes_{\varphi}C_{2}$ is isomorphic to $D_{4g}$ and there
is only a conjugacy class of orientation reversing involutions represented by
$z$.

The epimorphism $\theta^{\ast}:\Gamma^{\ast}\rightarrow G^{\ast}$ is:%
\[
a\rightarrow x;c_{0}\rightarrow xwx;c_{1}\rightarrow z;c_{2}\rightarrow w
\]

Assume that $g$ is odd. To use the formula of Gromadzky (\cite{G}) we need to
compute the centralizers:%
\begin{align*}
C(G^{\ast},\theta^{\ast}(c_{1}))  &  =C(G^{\ast},z)=\left\langle
z,(zw)^{g},(xz)^{g}\right\rangle \cong C_{2}\times C_{2}\times C_{2}\\
C(G^{\ast},\theta^{\ast}(c_{2}))  &  =C(G^{\ast},w)=\left\langle
w,(zw)^{g},(xw)^{g}\right\rangle \cong C_{2}\times C_{2}\times C_{2}.
\end{align*}
Now we have:%
\[
\left[  C(G^{\ast},\theta^{\ast}(c_{2})):\theta^{\ast}(C(\Gamma^{\ast}%
,c_{2}))\right]  =\left[  C(G^{\ast},\theta^{\ast}(c_{1})):\theta^{\ast
}(C(\Gamma^{\ast},c_{1}))\right]  =2
\]
The number of ovals of the involutions in the conjugacy classes $[z]$ and
$[w]$ is $2$.

The conjugacy classes $[(xz)^{g}]$ and $[(xw)^{g}]$ correspond to involutions
without ovals.

If $g$ is even we have:%
\begin{align*}
C(G^{\ast},\theta^{\ast}(c_{1}))  &  =C(G^{\ast},z)=\left\langle
z,(xz)^{2g}\right\rangle \cong C_{2}\times C_{2}\\
C(G^{\ast},\theta^{\ast}(c_{2}))  &  =C(G^{\ast},w)=\left\langle
w,(xz)^{2g}\right\rangle \cong C_{2}\times C_{2}.
\end{align*}
Therefore the number of ovals of the involutions in $[z]$ is:%
\[
\left[  C(G^{\ast},\theta^{\ast}(c_{0})):\theta^{\ast}(C(\Gamma^{\ast}%
,c_{0}))\right]  +\left[  C(G^{\ast},\theta^{\ast}(c_{1})):\theta^{\ast
}(C(\Gamma^{\ast},c_{1}))\right]  =2.
\]

Applying Theorem 3.3.2 of \cite{BCGG}, the topological types are:
$\{0,0,-2,-2\}$ if $g$ is odd and $\{-2\}$ if $g$ is even.
\end{proof}

\section{On the set of points with automorphism group of order $4g$ in the
moduli space of Riemann surfaces}

In this section we study family $\mathcal{F}_{g}$ of surfaces with $4g$
automorphims as subspace of the moduli space $\mathcal{M}_{g}$.

\begin{theorem} \label{familiaF}
The set of points $\mathcal{F}_{g}\subset\mathcal{M}_{g}$ corresponding to
Riemann surfaces of genus $g\geq2\,$given in Theorem \ref{Theorem 1} is 
the Riemann sphere with three punctures.
\end{theorem}

\begin{proof}
Let $\mathbf{T}_{g}$ be the Teichm\"{u}ller space of classes of surface
Fuchsian groups of genus $g$ and let $\pi:\mathbf{T}_{g}\rightarrow
\mathcal{M}_{g}$ be the canonical projection. The points in $\pi
^{-1}(\mathcal{F}_{g})$ are classes of surface Fuchsian groups contained in
Fuchsian groups with signature $(0;+;[2,2,2,2g])$. By Theorem \ref{Theorem 1},
up to automorphisms of Fuchsian groups and dihedral groups, there is only one
possible normal inclusion of surface groups of genus $g$ in groups with
signature $(0;+;[2,2,2,2g])$, this inclusion produces $i_{\ast}:\mathbf{T}
_{(0;+;[2,2,2,2g])}\rightarrow\mathbf{T}_{g}$ and $\pi\circ i_{\ast
}(\mathbf{T}_{(0;+;[2,2,2,2g])})\supset\mathcal{F}_{g}$. The set $i_{\ast
}(\mathbf{T}_{(0;+;[2,2,2,2g])})$ is an open disc and the map $\pi\circ
i_{\ast}\mid_{(\pi\circ i_{\ast})^{-1}(\mathcal{F}_{g})}$ is the projection
given by the action of a properly discontinuous group, then $\mathcal{F}_{g}$
is a real non-compact Riemann surface (a Riemann surface with punctures).
The family $\mathcal{F}_{g}$ can be identified with the space of orbifolds with 
three conic points of order 2, and one of order $2g$. One conic point of order 
2 corresponds to the conjugacy class of  the central involution $D^{g}$ in 
$D_{2g} =\langle A, D / A^2 =D^{2g}=(DA)^2=1 \rangle$, 
the conic point of order $2g$ corresponds to the conjugacy class of $D$. 
As the morphism $\theta$ in \ref{Theorem 1} is an epimorphism a second conic 
point of order 2 corresponds to the conjugacy class of $A$. Using a
M\"{o}bius transformation we can assume that the two order 2 conic points
are $0$ (corresponding to $[D^{g}]$)$,1$ (corresponding to $[A]$), and the 
conic point of order $2g$ is $\infty$ (corresponding to $[D]$). Then $\mathcal{F}_{g}$
 is parametrized by the position of the third conic point of order $2$. Hence $\mathcal{F}_{g}$
  is the Riemann sphere with three punctures.

The Riemann surface $\mathcal{F}_{g}$ admits an anticonformal involution whose
fixed point set is formed by the real Riemann surfaces in $\mathcal{F}_{g}$.
\end{proof}

\begin{theorem}
The real Riemann surface $\mathcal{F}_{g}$ has an anticonformal involution
whose fixed point set consists of three arcs $a_{1},a_{2},b$, corresponding to
the real Riemann surfaces in the family. The topological closure of
$\mathcal{F}_{g}$ in $\widehat{\mathcal{M}_{g}}$ has an anticonformal
involution whose fixed point set $\overline{a_{1}\cup a_{2}\cup b}$ (closure
of $a_{1}\cup a_{2}\cup b$ in $\widehat{\mathcal{M}_{g}}$) is a closed Jordan
curve. The set $\overline{a_{1}\cup a_{2}\cup b}\smallsetminus(a_{1}\cup
a_{2}\cup b)$ consists of three points: two nodal surfaces and the Wiman
surface of type II.
\end{theorem}

\begin{proof}
In $\mathcal{F}_{g}$, the surfaces have exactly $4g$ automorphisms, therefore
to complete $\mathcal{F}_{g}$ to $\overline{\mathcal{F}}_{g}$ (the topological
closure of $\mathcal{F}_{g}$ in $\widehat{\mathcal{M}_{g}}$), it is necessary
to add surfaces with more than $4g$ automorphisms and nodal surfaces in
$\widehat{\mathcal{M}_{g}}\smallsetminus\mathcal{M}_{g}$.

The surfaces in $\mathcal{F}_{g}$ having anticonformal automorphisms
correspond to the two inclusions $i_{1}$ and $i_{2}$ of Fuchsian groups with
signature $(0;+;[2,2,2,2g])$ in NEC groups with signature
$(0;+;[-];\{(2,2,2,2g)\})$ and the inclusion $j$ in NEC groups with signature
$(0;+;[2];\{(2,2g)\}),$ see Theorem \ref{Signaturas NEC}. The set of points in
$\mathcal{F}_{g}$ having anticonformal involutions are the following subsets
of $\mathcal{M}_{g}$:
\begin{align*}
\pi\circ i_{\ast}\circ i_{1\ast}(\mathbf{T}_{(0;+;[-];\{(2,2,2,2g)\})})  &
=a_{1}\\
\pi\circ i_{\ast}\circ i_{2\ast}(\mathbf{T}_{(0;+;[-];\{(2,2,2,2g)\})})  &
=a_{2}\\
\pi\circ i_{\ast}\circ j_{\ast}(\mathbf{T}_{(0;+;[2];\{(2,2g)\})})  &  =b.
\end{align*}

Since $\mathbf{T}_{(0;+;[-];\{(2,2,2,2g)\})}$ and $\mathbf{T}
_{(0;+;[2];\{(2,2g)\})}$ are of real dimension $1$, by Proposition
\ref{Conected} the sets $a_{1},a_{2}$ and $b$ are connected $1$-manifolds. Let
$\overline{a}_{1},\overline{a}_{2}$ and $\overline{b}$ be the closures of
$a_{1},a_{2}$ and $b$ in $\widehat{\mathcal{M}}_{g}$. Now we shall describe
the surfaces in $\overline{a_{1}\cup a_{2}\cup b}\smallsetminus(a_{1}\cup
a_{2}\cup b)$.

The arc $a_{1}\,$contains the surfaces with anticonformal involutions of
topological types $\{+2,0,-2,-2\}$ or $\{+1,0,-1,-3\}$, $a_{2}$ the surfaces
with anticonformal involutions of topological types $\{-1,-1,-g,-g\}$, and $b$
the surfaces with anticonformal involutions of topological types
$\{0,0,-2,-2\}$ or $\{-2\}$.

Let $\mathcal{G}$ be the graph of a nodal surface in the clousure of
$\mathcal{F}_{g}$ in $\widehat{\mathcal{M}}_{g}$. By the main theorem of
\cite{CGA} the graph $\mathcal{G}$ has $[D_{2g}:H(\theta\circ\delta)]$
vertices, where $\theta:\Gamma^{\prime}\rightarrow D_{2g}$ is the epimorphism
given in Theorem \ref{Theorem 1}, $\delta$ is an automorphism of the group
$\Gamma^{\prime}$ with signature $(0;+;[2,2,2,2g])$ and $H(\theta\circ
\delta)=\left\langle \theta\circ\delta(x_{1}x_{2}),\theta\circ\delta
(x_{3}),\theta\circ\delta(x_{4})\right\rangle $.

First of all, we shall consider the nodal surfaces in the clousure of the arc
$a_{1}$. Let $X=\mathbb{D}/\Gamma$ be a Riemann surface in the arc $a_{1}$.
Then there is an NEC group $\Gamma^{\ast}$ of signature
$(0;+;[-];\{(2,2,2,2g)\})$ such that $Aut^{\pm}(X)\cong\Gamma^{\ast}/\Gamma$.
The group $Aut^{\pm}(X)$ is isomorphic to:
\[
D_{2g}\times C_{2}=\left\langle w,x:w^{2}=x^{2}=(wx)^{2g}=1\right\rangle
\times\left\langle y:y^{2}=1\right\rangle
\]
and the epimorphism $\theta_{1}^{\ast}:\Gamma^{\ast}\rightarrow\Gamma^{\ast
}/\Gamma\cong D_{2g}\times C_{2}$ is defined in a canonical presentation of
$\Gamma^{\ast}$ by:
\[
\theta_{1}^{\ast}(c_{0})=x,\theta_{1}^{\ast}(c_{1})=y,\theta_{1}^{\ast}%
(c_{2})=(wx)^{g}w,\theta_{1}^{\ast}(c_{3})=w.
\]
The restriction $\theta_{1}^{\ast+}$ of $\theta_{1}^{\ast}$ to $(\Gamma^{\ast
})^{+}$ is
\begin{align*}
\theta_{1}^{\ast}(c_{0}c_{1})  &  =\theta_{1}^{\ast+}(x_{1})=xy,\theta
_{1}^{\ast}(c_{1}c_{2})=\theta_{1}^{\ast+}(x_{2})=y(wx)^{g}w,\\
\theta_{1}^{\ast}(c_{2}c_{3})  &  =\theta_{1}^{\ast+}(x_{3})=(wx)^{g}%
,\theta_{1}^{\ast}(c_{3}c_{0})=\theta_{1}^{\ast+}(x_{3})=wx.
\end{align*}

The nodal surfaces that are limits of the real surfaces in the arc $a_{1}$ are
given by automorphisms $\delta$ of the group $(\Gamma^{\ast})^{+}$ such that
$\theta\circ\delta$ is $\theta_{1}^{\ast+}\circ\gamma$, where $\gamma$ is an
automorphism of the group $\Gamma^{\ast}$. This fact reduces the possible
graphs of the nodal surfaces in $\overline{a_{1}}\cap(\widehat{\mathcal{M}%
}_{g}\smallsetminus\mathcal{M}_{g})$ to two: $\mathcal{G}^{1}(\theta)$ and
$\mathcal{G}^{2}(\theta)$. If $H^{1}(\theta_{1}^{\ast})$ is the subgroup of
$D_{2g}$ generated by $\theta_{1}^{\ast+}(x_{1}x_{2})$, $\theta_{1}^{\ast
+}(x_{3})$, $\theta_{1}^{\ast+}(x_{4})$, the number of vertices of
$\mathcal{G}^{1}(\theta)$ is given by $[D_{2g}:H^{1}(\theta_{1}^{\ast})]$, and
for $\mathcal{G}^{2}(\theta)$ the number of vertices is given by the index of
the subgroup $H^{2}(\theta_{1}^{\ast})$ of $D_{2g}$ generated by $\theta
_{1}^{\ast+}(x_{1})$, $\theta_{1}^{\ast+}(x_{2}x_{3})$, $\theta_{1}^{\ast
+}(x_{4})$. Hence the number of components (vertices of the corresponding
graphs) of such nodal surfaces are, respectively:
\begin{align*}
\lbrack D_{2g} &  :\left\langle \theta_{1}^{\ast}(c_{0}c_{2}),\theta_{1}
^{\ast}(c_{2}c_{3}),\theta_{1}^{\ast}(c_{3}c_{0})\right\rangle ]=\\
\lbrack D_{2g} &  :\left\langle (wx)^{g-1},(wx)^{g},wx\right\rangle
]=[D_{2g}:\left\langle wx\right\rangle ]=2
\end{align*}
and
\begin{align*}
\lbrack D_{2g} &  :\left\langle \theta_{1}^{\ast}(c_{0}c_{1}),\theta_{1}
^{\ast}(c_{1}c_{3}),\theta_{1}^{\ast}(c_{3}c_{0})\right\rangle ]=\\
\lbrack D_{2g} &  :\left\langle xy,yw,wx\right\rangle ]=1.
\end{align*}

By the main theorem of \cite{CGA} the degree of the vertices of the graph
$\mathcal{G}^{1}(\theta_{1}^{\ast})$ is $[H^{1}(\theta_{1}^{\ast
}):\left\langle \theta_{1}^{\ast}(x_{1}x_{2})\right\rangle ]$. Since
$[H^{1}(\theta_{1}^{\ast}):\left\langle \theta_{1}^{\ast}(x_{1}x_{2}
)\right\rangle ]=1$ if $g$ is even and $2$ if $g$ is odd, the vertices of the
graph $\mathcal{G}^{1}(\theta_{1}^{\ast})$ have degree $1$ or $2$, and the
graph has two vertices and one or two edges joining them, the graph
$\mathcal{G}^{1}(\theta_{1}^{\ast})$ is a 1- or 2-dipole.

By \cite{CGA} and since $[H^{2}(\theta_{1}^{\ast}):\left\langle \theta
_{1}^{\ast}(x_{2}x_{3})\right\rangle ]=g$, the graph $\mathcal{G}^{2}
(\theta_{1}^{\ast})$ has one vertex and $g$ loops. We call $X_{D}$ the nodal
surface corresponding to $\mathcal{G}^{1}(\theta_{1}^{\ast})$ and $X_{R}$ the
nodal surface corresponding to $\mathcal{G}^{2}(\theta_{1}^{\ast})$.

Each vertex of $\mathcal{G}^{i}(\theta_{1}^{\ast})$ corresponds to one
component of the nodal surface. The uniformization groups of the components of
$X_{D}$ and $X_{R}$ are $\ker\omega_{1}$ and $\ker\omega_{2}$ respectively,
where the homorphisms $\omega_{i}:\widehat{\Gamma}\rightarrow D_{2g}$, $i=1,2$
are defined by:
\begin{align*}
\omega_{1}  &  :\gamma_{1}\rightarrow\theta_{1}^{\ast}(c_{0}c_{2})=\theta
_{1}^{\ast}(x_{1}x_{2})=x(wx)^{g}w,\\
\gamma_{2}  &  \rightarrow\theta_{1}^{\ast}(c_{2}c_{3})=\theta_{1}^{\ast
}(x_{3})=(wx)^{g},\\
\gamma_{3}  &  \rightarrow\theta_{1}^{\ast}(c_{3}c_{0})=\theta_{1}^{\ast
}(x_{4})=wx.
\end{align*}
from a Fuchsian group $\widehat{\Gamma}$ with signature $(0;+;[\infty,2,2g])$
(one parabolic class of transformations) and presentation $\left\langle
\gamma_{i}:\gamma_{1}\gamma_{2}\gamma_{3}=\gamma_{2}^{2}=\gamma_{2}%
^{2g}\right\rangle $. As a consequence each component of $X_{D}$ has genus
$\frac{g}{2}$ if $g$ is even and $\frac{g-1}{2}$ if $g$ is odd.

Now
\begin{align*}
\omega_{2}  &  :\gamma_{1}\rightarrow\theta_{1}^{\ast}(c_{0}c_{1})=\theta
_{1}^{\ast}(x_{1})=xy,\\
\gamma_{2}  &  \rightarrow\theta_{1}^{\ast}(c_{1}c_{3})=\theta_{1}^{\ast
}(x_{2}x_{3})=yw,\\
\gamma_{3}  &  \rightarrow\theta_{1}^{\ast}(c_{3}c_{0})=\theta_{1}^{\ast
}(x_{4})=wx
\end{align*}
where $\widehat{\Gamma}$ has presentation $\left\langle \gamma_{i}:\gamma
_{1}\gamma_{2}\gamma_{3}=\gamma_{1}^{2}=\gamma_{2}^{2g}\right\rangle $ and
signature $(0;+;[\infty,2,2g])$. The component of $X_{R}$ has genus $0$.

The set $\overline{a}_{1}$ intersects $\widehat{\mathcal{M}}_{g}
\smallsetminus\mathcal{M}_{g}$ in two points to : $X_{D}$ and $X_{R}$, thus
$a_{1}$ is an arc.

Let $X_{8g}$ be the Wiman surface of type II with automorphism group of order
$8g$ (for $g=2$, $\mathrm{Aut}(X_{16})=GL(2,3)$) and such that the signature
of the Fuchsian group $\Delta$ uniformizing $X_{8g}/\mathrm{Aut}^{\pm}
(X_{8g})$ is $(0;+;[-];\{(2,4,4g)\})$ (signature $(0;+;[-];\{(2,3,8)\})$ for
$g=2$). The surface $X_{8g}$ belongs to the clousure of the arc $a_{2}$ since
a group of signature $(0;+;[-];\{(2,2,2,2g)\})$ is contained in $\Delta$ and
the epimorphism $\theta_{2}^{\ast}$ may be extended to $\Delta$.

There is also one point in $\overline{a_{2}}\cap(\widehat{\mathcal{M}_{g}
}\smallsetminus\mathcal{M}_{g})$. The graph of $\ \overline{a_{2}}
\cap(\widehat{\mathcal{M}_{g}}\smallsetminus\mathcal{M}_{g})$ has only one
vertex since:%
\[
\theta_{2}^{\ast}(c_{0})=x,\theta_{2}^{\ast}(c_{1})=y,\theta_{2}^{\ast}
(c_{2})=y(wx)^{g},\theta_{2}^{\ast}(c_{3})=w
\]
\[
\theta_{2}^{\ast}(c_{0}c_{1})=xy,\theta_{2}^{\ast}(c_{1}c_{2})=(wx)^{g}
,\theta_{2}^{\ast}(c_{2}c_{3})=y(wx)^{g}w,\theta_{2}^{\ast}(c_{3}c_{0})=wx
\]
\[
\theta_{2}^{\ast}(c_{0}c_{2})=xy(wx)^{g},\theta_{2}^{\ast}(c_{2}
c_{3})=y(wx)^{g}w,\theta_{2}^{\ast}(c_{3}c_{0})=wx
\]
\[
\theta_{2}^{\ast}(c_{0}c_{1})=xy,\theta_{2}^{\ast}(c_{1}c_{3})=yw,\theta
_{2}^{\ast}(c_{3}c_{0})=wx
\]
and $\left\langle yx(wx)^{g},y(wx)^{g}w,wx\right\rangle =\left\langle
yx,yw,wx\right\rangle \cong D_{2g}$. Hence $X_{R}\in\overline{a_{2}}
\cap(\widehat{\mathcal{M}_{g}}\smallsetminus\mathcal{M}_{g})$. Therefore
$\overline{a}_{2}\smallsetminus a_{2}$ has two points: $X_{R}$ and $X_{8g}$,
thus $a_{2}$ is an arc.

Finally, in a similar way one sees that $b$ joins $X_{D}$ and $X_{8g}$, so $b$
is an arc and $\overline{a_{1}\cup a_{2}\cup b}$ is a closed Jordan curve, the
fixed point set of an anticonformal involution of $\overline{\mathcal{F}}_{g}$.
\end{proof}

\begin{remark}
The surfaces in the arc $a_{2}$ are the surfaces having maximal number of
ovals among the Riemann surfaces of genus $g$ with four non-conjugate
anticonformal involutions, two of which do not commute (see Theorem 1 in
\cite{IG}).
\end{remark}

\begin{remark}
As a consequence of the above theorem we have that $\overline{\mathbb{R}%
\mathcal{F}_{g}}\cap\mathcal{M}_{g}$ has two connected components, then we
cannot always continuously deform a real algebraic curve with $4g$
automorphisms to another real algebraic curve with the same characteristics
mantaining the real character and the number of automorphisms along the path.
\end{remark}

\end{document}